\documentclass[10pt]{amsart}
\usepackage{amsmath, amssymb, amscd, amsthm, mathrsfs, url ,pinlabel, verbatim, lipsum, wrapfig}
\usepackage[text={6.5in,9.5in}, centering, letterpaper, dvips]{geometry}
\usepackage{color, multirow, dcpic, latexsym, pictexwd, graphicx, epstopdf, tikz-cd, wrapfig, float, xcolor}
\usepackage{relsize, latexsym, latexsym, tikz, mdframed, float, mathtools, flowchart, multicol, stmaryrd, caption, capt-of, old-arrows, dirtytalk, footmisc, setspace, enotez, tabularx, pifont, tensor, csquotes}
\usepackage{svg}
\usepackage{fancyhdr, tikz, tikz-cd, setspace, tensor, enumitem, dirtytalk, nicematrix}
\usepackage{mathabx}

\usepackage{pst-node, auto-pst-pdf}
\usetikzlibrary{positioning, arrows.meta, knots, braids}

\usepackage{scalerel,stackengine}
\stackMath
\newcommand\reallywidehat[1]{%
\savestack{\tmpbox}{\stretchto{%
  \scaleto{%
    \scalerel*[\widthof{\ensuremath{#1}}]{\kern-.6pt\bigwedge\kern-.6pt}%
    {\rule[-\textheight/2]{1ex}{\textheight}}%WIDTH-LIMITED BIG WEDGE
  }{\textheight}% 
}{0.5ex}}%
\stackon[1pt]{#1}{\tmpbox}%
}

\usepackage[backref=page, linktocpage=true]{hyperref}

\renewcommand*{\backref}[1]{}
\renewcommand*{\backrefalt}[4]{%
    \ifcase #1 (Not cited.)%
    \or        (Cited on page~#2.)%
    \else      (Cited on pages~#2.)%
    \fi}

\definecolor{maroon}{rgb}{0.5, 0.0, 0.0}
\definecolor{darkblue}{rgb}{0.03, 0.27, 0.49}

\hypersetup{
  colorlinks   = true, %Colours links instead of ugly boxes
  urlcolor     = maroon, %Colour for external hyperlinks
  linkcolor    = darkblue, %Colour of internal links
  citecolor   = maroon %Colour of citations
}

\usepackage[colorinlistoftodos,textwidth=3cm]{todonotes}

\usepackage{indentfirst}
\usepackage{pinlabel}	
\newtheorem{proposition}{Proposition}[section]
\newtheorem{theorem}[proposition]{Theorem}
\newtheorem{corollary}[proposition]{Corollary}
\newtheorem{lemma}[proposition]{Lemma}

\newtheorem*{theorem*}{Theorem}
\newtheorem*{proposition*}{Proposition}
\newtheorem*{lemma*}{Lemma}
\newtheorem*{corollary*}{Corollary}
\newtheorem*{problem*}{Problem}
\newtheorem{question*}{Question}
\newtheorem*{rep@theorem}{\rep@title}
\newcommand{\newreptheorem}[2]{
\newenvironment{rep#1}[1]{
 \def\rep@title{#2 \ref{##1}}
 \begin{rep@theorem}}
 {\end{rep@theorem}}}
\makeatother

\newtheorem*{rep@proposition}{\rep@title}
\newcommand{\newrepproposition}[2]{
\newenvironment{rep#1}[1]{
 \def\rep@title{#2 \ref{##1}}
 \begin{rep@proposition}}
 {\end{rep@proposition}}}
\makeatother

\theoremstyle{definition}
\newtheorem{definition}[proposition]{Definition}
\newtheorem{question}[proposition]{Question}

\newreptheorem{theorem}{Theorem}
\newreptheorem{lemma}{Lemma}
\newreptheorem{proposition}{Proposition}
\newreptheorem{corollary}{Corollary}
\newreptheorem{definition}{Definition}

\newcommand{\Q}{\mathbb{Q}}

\newcommand{\Z}{\mathbb{Z}}
\newcommand{\R}{\mathbb{R}}

\newcommand{\bunderline}[1]{\underline{#1\mkern-2mu}\mkern2mu }

\def\du {\bar{d}}
\def\dl {\bunderline{d}}

\begin{document}
\title{On homology spheres of the trivial local equivalence class}

\author{Jaewon Lee}
\address{Department of Mathematical Sciences, KAIST, 34141 Daejeon, Republic of Korea}
\email{\url{freejw@kaist.ac.kr}}
\urladdr{\url{https://mathsci.kaist.ac.kr/~freejw}}

\author{O{\u{g}}uz \c{S}avk}
\address{Department of Mathematics, Middle East Technical University, 06800 \c{C}ankaya, Ankara, Turkey}
\email{\url{savk@metu.edu.tr}}
\urladdr{\url{https://sites.google.com/view/oguzsavk}}

\date{}

\begin{abstract} 
  In the homology cobordism group $\Theta_\mathbb{Z}^3$, it is not known if there are non-trivial linear dependences between Seifert fibered spheres. Based on involutive Heegaard Floer theory, Hendricks, Manolescu, and Zemke introduced the local equivalence group $\mathfrak{I}$ along with the homomorphism $h:\Theta_\mathbb{Z}^3 \rightarrow \mathfrak{I}$. Using the work of Dai and Stoffregen, one can find non-trivial linear dependences between the images of Seifert fibered spheres under $h$. Therefore, it is interesting to ask if such dependences in $\mathfrak{I}$ originate from $\Theta_\mathbb{Z}^3$. In this paper, by employing the $r_s$-invariants from the filtered instanton Floer homology developed by Nozaki, Sato, and Taniguchi, we provide certain conditions to guarantee that such relations are not realized even in the rational homology cobordism group. We also discuss the local equivalence class of the $\operatorname{Pin(2)}$-equivariant Seiberg--Witten Floer stable homotopy type.
\end{abstract}

\maketitle
\section{Introduction}
\label{sec:intro}

The homology cobordism group $\Theta^3_\Z$ of homology spheres, introduced by Gonz{\'a}lez-Acu{\~n}a \cite{GA70}, has played a central role in the development of low-dimensional topology over the last five decades \cite{Man18, Hom23, Sav24}. The algebraic structure of $\Theta^3_\Z$ is linked to several topological problems such as the three fibers conjecture \cite{Law88, Kol08}, the disproof of the triangulation conjecture for topological manifolds in high dimensions \cite{Man16}, and so on. While it is known that the homology cobordism group contains an infinite rank summand by \cite{DHST23}, the existence of a torsion element is one of the most notable open problems about the structure of $\Theta^3_\Z$.

Seifert fibered spheres have been core objects to study the structure of $\Theta^3_\Z$. In particular, the subgroup $\Theta^3_{SF}$ generated by Seifert fibered spheres itself contains an infinite rank subgroup in $\Theta^3_\Z$ \cite{Fur90, FS90}. Moreover, the aforementioned infinite rank summand \cite{DHST23} lies also in $\Theta_{SF}^3$. On the other hand, there are several infinite families of Seifert fibered spheres known to bound contractible manifolds or homology balls \cite{AK79, CH81, FS81, Fic84}, namely trivial in $\Theta_\Z^3$. However, it is neither known in $\Theta_\Z^3$ whether there exists a non-trivial Seifert fibered sphere of finite order, nor, more generally, whether there exists any non-trivial linear dependence among them. More precisely, we ask:

\begin{question}\label{question-1}
  Are distinct Seifert fibered spheres always linearly independent in $\Theta_\Z^3$, provided they do not bound homology balls?
\end{question}

Hendricks, Manolescu, and Zemke \cite{HMZ18} defined the notion of the \textit{local equivalence class} of an involutive Heegaard Floer chain complex \cite{HM17} of a homology sphere. Such classes form an abelian group $\mathfrak{I}$, called the \textit{local equivalence group}, together with the homomorphism: $$h:\Theta_\Z^3\rightarrow \mathfrak{I}.$$

Let $\mathfrak{I}_{SF}$ be the image of $\Theta_{SF}^3$ under $h$. While $\mathfrak{I}_{SF}$ has been extensively studied in \cite{DM19, DS19} and is known to be isomorphic to $\Z^\infty$, it is nevertheless possible to find many linear dependences in $\mathfrak{I}_{SF}$. Therefore, this fact and Question \ref{question-1} motivate us to study the kernel of the restriction of $h$ to $\Theta^3_{SF}$. For example, the Brieskorn spheres $\Sigma(2, 3, 12n-7)$ for $n\ge 1$ are known to share the same local equivalence class \cite{HM17}, yet they were previously shown by Furuta \cite{Fur90} to be linearly independent in $\Theta_\Z^3$.

Dai and Manolescu \cite{DM19} introduced a notion called \emph{monotone graded subroots} of Seifert fibered spheres to simply determine their local equivalence classes. By analyzing monotone graded subroots systematically, Dai and Stoffregen \cite{DS19} proved that: $$\mathfrak{I}_{SF}  = \langle h(B(n)) \rangle \cong \Z^\infty,$$ where $B(0)$ is the Poincar\'e homology sphere $\Sigma(2,3,5)$ and $B(n)$ is the Brieskorn sphere $\Sigma(2n+1, 4n+1, 4n+3)$ with $n > 0$. 

Therefore, potential counterexamples to Question \ref{question-1} arise from linear dependence relations with respect to this explicit basis $\{h(B(n))\}_{n\ge 0}$ of $\mathfrak{I}_{SF}$. For any Seifert fibered sphere $Y$, there exists a unique homology sphere $Z_Y$ as the linear expression of $Y$ with respect to $\{B(n)\}$ such that $h(Y)=h(Z_Y)$, or equivalently, $$Y\# -Z_Y \in \operatorname{ker}(h).$$ We call such $Z_Y$ the \emph{associated homology sphere to $Y$ with respect to $\{B(n)\}$}. For example, for any $Y$ in the aforementioned family $\Sigma(2,3,12n-7)$, the associated homology sphere $Z_Y$ is $\Sigma(2,3,5)=B(0)$.

In this paper, we find certain conditions for a Seifert fibered sphere $Y$ to guarantee that $Y$ and the associated homology sphere $Z_Y$ are not anymore related in $\Theta_\Z^3$. Such conditions are expressed in terms of the product $\mathcal{E} (Y)$ of the exponents of $Y$ and the Fintushel--Stern $R$-invariant \cite{FS85}. Note that any linear combinations of the homology spheres of the form $Y\# -Z_Y$ are potential counterexamples to Question \ref{question-1}. We also prove their linear independence.

\begin{theorem}\label{thm-Z}
     Let $Y$ be a Seifert fibered sphere and let~$Z_Y$ be the associated homology sphere of the same image in $\mathfrak{I}$. If~$R(Y) > 0$ and $\mathcal{E} (Y)$ is distinct from $30$ and $(2m+1)(4m+1)(4m+3)$, then $Y\# -Z_Y$ is non-trivial in $\ker{h}$. Moreover, given such a family of homology spheres $Y(n)$ with the associated $Z(n)$, if all $\mathcal{E} (Y(n))$ are distinct, then $Y(n)\#-Z(n)$ are linearly independent in $\operatorname{ker} (h)$ up to rational homology cobordism.
\end{theorem}

Note that the previously known examples, such as $\{\Sigma(2, 3, 12n-7)\}$, also fall under Theorem~\ref{thm-Z}. We remark that we obtain linear independence not only in $\Theta_\Z^3$, but also in the \textit{rational homology cobordism group} $\Theta_\Q^3$ of rational homology spheres. Together with the assumption on $\mathcal{E}(Y)$ in Theorem~\ref{thm-Z}, Furuta's result \cite{Fur90} is sufficient to prove linear independence in $\Theta^3_\Z$ of such homology spheres $Y(n)\# -Z(n)$. However, this does not directly imply linear independence in $\Theta_\Q^3$, since the kernel of $\Theta_\Z^3\rightarrow\Theta_\Q^3$ is known to be non-trivial \cite{FS84, AL18, Sav20, Sim21}.

To obtain their linear independence in $\Theta^3_\Q$, we employ the $r_s$-invariants from the filtered instanton Floer homology, introduced by Nozaki, Sato, and Taniguchi \cite{NST24}. They extended the previous work of Donaldson \cite{Don02}, Fintushel and Stern \cite{FS85, FS90}, and Furuta \cite{Fur90} in instanton Floer theory.

Note that Theorem \ref{thm-Z} indicates that the $d$-invariant \cite{OS03a}, the $\dl$- and $\du$-invariants \cite{HM17}, and the $\phi_n$-invariants for all $n\ge 1$ \cite{DHST23} must vanish for homology spheres $Y(n)\# -Z(n)$ in Theorem \ref{thm-Z}. In other words, for any homology sphere $Y$ obtained by a non-trivial linear combination of $Y(n)\# -Z(n)$, we see that: $$\underline{d} (Y) = d(Y) = \overline{d}(Y) = \phi_n(Y) = 0.$$ We observe that the $d$-invariant is not a complete rational homology cobordism invariant. One can find some homology spheres of infinite order in $\Theta^3_\Q$ with vanishing $d$-invariants from \cite{NST24}. We also note that the $\bar{\mu}$-invariant \cite{Neu80, Sie80} vanishes by the result of \cite{DS19}.

On the other hand, Stoffregen \cite{Sto20} defined an analogous group in Seiberg--Witten Floer theory together with the homomorphism: $$\Theta^3_\Z \to \mathfrak{LE},$$ where the latter group denotes the local equivalence group of the $\operatorname{Pin}(2)$-equivariant Seiberg--Witten Floer stable homotopy type. Recently, Dai, Sasahira, and Stoffregen \cite{DSS23} proved that two Seifert fibered spheres have the same image in $\mathfrak{LE}$ if and only if they have the same image in $\mathfrak{I}$. Therefore, for a Seifert fibered sphere $Y$, the homology sphere $Y\# -Z_Y$ lies in $\operatorname{ker}(\Theta^3_\Z \to \mathfrak{LE})$ when the associated homology sphere $Z_Y$ to $Y$ is also Seifert fibered. Moreover, they also fall under Theorem \ref{thm-Z} as well.

For those homology spheres, the $\kappa$-invariant \cite{Man14} and the $\kappa o_i$-invariants for $i=0, \ldots, 7$ \cite{Lin15} from Seiberg--Witten Floer stable homotopy type, and the $\delta$-invariant \cite{Fro10}, the $\alpha$-, $\beta$-, and $\gamma$-invariants \cite{Man16}, and the $\underline{\delta}$- and $\overline{\delta}$-invariants \cite{Sto17b} from Seiberg--Witten Floer homology must vanish. Namely, for any homology spheres $Y$ obtained by non-trivial linear combinations of $Y(n)\# -Z(n)$, we also have the following: $$ \kappa (Y)= \kappa o_i (Y) =  \alpha (Y)= \beta (Y) = \gamma (Y) = \underline{\delta} (Y) = \delta (Y) = \overline{\delta} (Y) = 0.$$

Based on \cite{DS19}, we construct explicit examples of Theorem \ref{thm-Z} by taking $\{Y(n)\}$ as the families of the Brieskorn spheres: $$Y_1(n) = \Sigma(4n+1, 6n+2, 12n+1) \quad \text{and} \quad Y_2(n)=\Sigma(4n-1, 6n-2, 12n-1).$$ Their monotone graded subroots and hence their local equivalence classes were computed by Karakurt and the second author \cite{KS22}. By using the fact that every monotone graded subroot is decomposed into simpler ones~\cite[Theorem~1.1]{DS19}, we first find associated homology spheres $Z_1 (n)$ and $Z_2 (n)$ to $Y_1(n)$ and $Y_2(n)$, respectively. From these linear relations, we obtain infinitely many homology spheres lying in $\operatorname{ker}(h)$ as follows: $$ \big \{Y_1(n) \ \# \ -Z_1(n)  \big \}_{n\geq 1}  \cup \big \{Y_2(n) \ \# \ -Z_2(n)  \big \}_{n \geq 1}.$$

These families satisfy the conditions of Theorem~\ref{thm-Z} so that we obtain a substantial structural difference between $\Theta^3_{SF}$ and $\mathfrak{I}_{SF}$, providing affirmative evidence for Question \ref{question-1}. In other words, those families reprove an implicitly known result from other examples in several literature \cite{Fur90, HM17, NST24}.

\begin{corollary}\label{cor-A}
There are infinitely many homology spheres of the trivial local equivalence class that are linearly independent in the rational homology cobordism group.
\end{corollary}
    
Similarly, we provide examples of Theorem \ref{thm-Z} whose local equivalence classes of Seiberg--Witten Floer stable homotopy type are also the same, by using the following family of the Brieskorn spheres: $$Y_3(n) = \Sigma(8n+1, 12n+1, 24n+5).$$ Relying on \cite{KS22}, we find their associated homology spheres $Z_3 (n)$ that are Seifert fibered, so that the resulting homology spheres are as follows: $$\{Y_3(n)\ \# \ -Z_3(n)\}_{n\geq1}.$$

To the authors' best knowledge, it is not known whether there exists a Seifert fibered sphere of the trivial local equivalence class but non-trivial in $\Theta^3_\Q$. The reason is that it is difficult to compute the $r_s$-invariant in general when the Fintushel--Stern $R$-invariant \cite{FS85} is negative. Moreover, Issa and McCoy \cite{IM18} proved that for a Seifert fibered sphere $Y$ with $d(Y) =0$, $R(Y)$ is always negative. Therefore, we conclude the introduction with the following question:

\begin{question}
Is there a Seifert fibered sphere $Y$ with $d(Y)=0$ that does not bound a rational homology ball?
\end{question}

\subsection*{Conventions}
Every manifold is assumed to be compact, connected, oriented, and smooth. The connected sum of $n$ copies of a manifold $Y$ is denoted by $nY$. We use the terms \say{homology sphere} and \say{homology ball} to refer to an \say{integral homology $3$-sphere} and an \say{integral homology $4$-ball}, respectively. $Y$ denotes a homology sphere unless otherwise stated. Seifert fibered spheres are oriented as the links of Brieskorn--Hamm complete intersection singularities in $\mathbb{C}^3$, so they bound plumbed $4$-manifolds with unique minimal negative definite plumbing graphs of central weights $e\leq-1$, see \cite[Section~1.1]{Sav02}. $\mathbb{F}$ denotes the field of characteristic $2$. Our grading shift convention for involutive Heegaard Floer invariants is in line with the one in \cite{DS19}. 

\subsection*{Acknowledgments}
The authors would like to thank Marco Golla, JungHwan Park, and Masaki Taniguchi for helpful conversations. The authors are also grateful to Hayato Imori and Imogen Montague for reading the draft and providing detailed feedback. The first author is partially supported by the Samsung Science and Technology Foundation (SSTF-BA2102-02) and the NRF grant RS-2025-00542968. The second author was supported by the CNRS postdoctoral fellowship at the Laboratoire de Math\'ematiques Jean Leray in Nantes Universit\'e, France.

%-----Section 2-----

\section{Preliminaries on Floer Theories}

The proof of Theorem \ref{thm-Z} involves two theories: involutive Heegaard Floer homology \cite{HM17} and filtered instanton Floer homology \cite{NST24}. In this section, we briefly recall the basic concepts, key properties, and essential notions of each theory. For the first two subsections, we review the involutive Heegaard Floer homology and its local equivalence class in terms of monotone graded subroots. In the third, we consider an analogue in the Seiberg--Witten Floer theory. In the last, we briefly recall the filtered instanton homology and the construction of the $r_s$-invariants.

\subsection{Involutive Heegaard Floer homology and local equivalence group}

In this subsection, we review the involutive Heegaard Floer theory and the notion of local equivalence. We will only consider the minus flavor of the Heegaard Floer chain complex $CF^-(Y)$ \cite{OS04}, which is a graded $\mathbb{F}[U]$-module, where $U$ is a formal variable of grading $-2$.

Ozsv\'{a}th and Szab\'{o} \cite{OS04} defined the Heegaard Floer homology $HF^-(Y)$ for a closed 3-manifold $Y$, and proved that it is a diffeomorphism invariant. Hendricks and Manolescu \cite{HM17} defined a homotopy involution $\iota$ on a Heegaard Floer chain complex by composing the map which arises from exchanging $\alpha$-curves and $\beta$-curves in the Heegaard diagram with the canonical chain homotopy equivalence based on the naturality \cite{JTZ21}. Then they introduced the involutive Heegaard Floer homology $HFI^-(Y)$ as the homology of the mapping cone of $1+\iota$ on $CF^-(Y)$. When $Y$ is a homology sphere, $HFI^-(Y)$ gives rise to two homology cobordism invariants $\dl(Y)$ and $\du(Y)$.

Instead of computing $HFI^- (Y)$ directly, Hendricks, Manolescu, and Zemke \cite{HMZ18} defined an equivalence relation of the involutive Heegaard Floer chain complexes of homology spheres. The equivalence class is a homology cobordism invariant. Moreover, the classes form an abelian group $\mathfrak{I}$, called the \emph{local equivalence group}, together with the homomorphism $h:\Theta_\Z^3\rightarrow \mathfrak{I}$ as mentioned in Section~\ref{sec:intro}.

Let $Y_1$ and $Y_2$ be two homology spheres. Recall from \cite{HM17} that a cobordism $W$ from $Y_1$ to $Y_2$ with a self-conjugate spin$^c$-structure $\mathfrak{t}$ induces the $\iota$-equivariant $\mathbb{F}[U]$-chain map $F_{W, \mathfrak{t}}: CF^-(Y_1)\rightarrow CF^-(Y_2)$. Moreover, we have:

\begin{theorem}
  If two homology spheres $Y_1$ and $Y_2$ are homology cobordant, then there exist two grading-preserving $\iota$-equivariant $\mathbb{F}[U]$-chain maps $F:CF^-(Y_1)\rightarrow CF^-(Y_2)$ and $G:CF^-(Y_2)\rightarrow CF^-(Y_1)$ where $F$ and $G$ induce isomorphisms on homology after localization of $U$.
\label{thm-motivation}
\end{theorem}

The local equivalence group is modeled on the homology cobordism group in the sense of Theorem \ref{thm-motivation}.

\begin{definition}
Let $C$ be a free, $\Z$-graded and finitely generated chain complex over $\mathbb{F}[U]$ with a grading-preserving $\mathbb{F}[U]$-equivariant homotopy involution $\iota$. The pair $(C, \iota)$ is called an \textit{$\iota$-complex} if the localization of its homology is isomorphic to $\mathbb{F}[U, U^{-1}]$, i.e., $$U^{-1}H_*(C) \cong \mathbb{F}[U, U^{-1}].$$
\end{definition}

\noindent Note that the involutive Heegaard Floer chain complex $(CF^-(Y), \iota)$ for a homology sphere $Y$ is an $\iota$-complex.

\begin{definition}
  Let $(C_1, \iota_1)$ and $(C_2, \iota_2)$ be two $\iota$-complexes. If there are grading-preserving $\iota$-equivariant chain maps $F:C_1\rightarrow C_2$ and $G:C_2\rightarrow C_1$ which induce isomorphisms on the homology after localization of $U$, then two $\iota$-complexes are called \textit{locally equivalent}.
\end{definition}

It is proved in \cite{HMZ18} that the set of $\iota$-complexes modulo local equivalence with the operation tensor product forms an abelian group $\mathfrak{I}$, called the \emph{local equivalence group}. The identity element of $\mathfrak{I}$ is given by the trivial complex consisting of a single $\mathbb{F}[U]$-tower starting at grading zero, together with the trivial $\iota$. Inverse elements in $\mathfrak{I}$ are obtained by dualizing $\iota$-complexes. Moreover, Hendricks, Manolescu, and Zemke \cite{HMZ18} also proved that there is an $\iota$-equivariant chain homotopy equivalence: $$(CF^-(Y_1\#Y_2),\iota) \simeq ((CF^-(Y_1)\otimes CF^-(Y_2))[-2], \iota_1\otimes \iota_2),$$ where $[-2]$ denotes the grading shift by $-2$. By Theorem~\ref{thm-motivation}, the local equivalence class of $Y$ is a homology cobordism invariant, and from the connected sum formula above we have the homomorphism defined as: $$h: \Theta_\mathbb{Z}^3 \to \mathfrak{I}, \ \ h(Y) = (CF^-(Y)[-2], \iota),$$ where the $d$-invariant \cite{OS03a}, the $\dl$- and $\du$-invariants \cite{HM17}, and the $\phi_n$-invariants for all $n \geq 1$ \cite{DHST23} factor through $\mathfrak{I}$.

\subsection{Monotone graded subroots for Seifert fibered spheres}

In this subsection, we explain the roles of graded roots and their monotone graded subroots to describe local equivalence classes of Seifert fibered spheres. We follow the references \cite{Nem05, DM19, DS19, KS22}.

In \cite{Nem05}, N\'emethi introduced two combinatorial objects to compute Heegaard Floer homology of Seifert fibered spheres effectively. The first object is called a \emph{graded root} $(R,\nu)$ equipped with a grading function $\nu$ on the vertex set of $R$, where $R$ is an upwards-opening tree with an infinite downwards stem defined from the negative definite plumbing graph for a Seifert fibered sphere. The second one is a graded $\mathbb{F}[U]$-module, called the \emph{lattice homology} $\mathbb{H} ^- (R)$. It is combinatorially defined from $(R,\nu)$, whose grading is induced by $\nu$ with the formal variable $U$ of grading $-2$. The grading on $\mathbb{H}^-(R)$ is denoted by the same symbol $\nu$ by abuse of notation. See \cite{Nem05, DM19} for precise definitions.

As observed in \cite{Dai18, DM19}, for a Seifert fibered sphere, there is an involution $J$ on the vertex set of $R$ that reflects a graded root along the central vertical axis. We call the triple 
$(R,\nu,J)$ a \textit{symmetric graded root}, and an example is depicted in Figure~\ref{fig:graded_root}. Moreover, there is an $\mathbb{F}[U]$--equivariant action $J_*$ induced by $J$ on $\mathbb{H} ^- (R)$. See \cite[Section~2.1]{Dai18} for further discussion.

\begin{figure}[htbp]
\centering
\includegraphics[width=0.3\columnwidth]{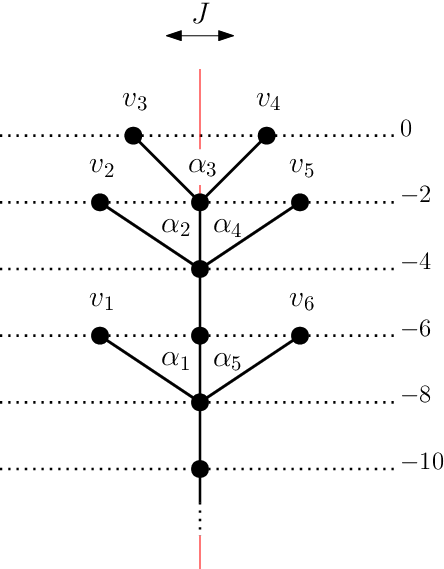}
\caption{The graded root $R$ with the involution $J$ for the Brieskorn sphere $\Sigma(3,4,13)$. The leaves and angles of $R$ are labeled by $v_1, \ldots v_6$ and $\alpha_1, \ldots \alpha_5$, respectively.}
\label{fig:graded_root}
\end{figure}

By combining the results of Ozsv\'ath and Szab\'o \cite{OS03b}, N\'emethi \cite{Nem05}, and Dai and Manolescu \cite{DM19}, we have the following graded $\mathbb{F}[U]$-module isomorphism: $$F: HF^-(Y) \xrightarrow{\cong} \mathbb{H}^-(R)[\sigma] \quad \text{and} \quad F \circ \iota_* =  J_* \circ F,$$ where $\sigma$ denotes a grading shift\footnote{For a Seifert fibered sphere, this shift is determined by algebraic topology of the minimal negative definite plumbing bounded by $Y$. See \cite[Section~2.3]{KS22} and references therein for a precise definition.} and $\iota_*$ is the induced involution by $\iota$ on $HF^-(Y)$.

Following \cite{DM19}, we next construct a finitely generated free graded $\mathbb{F}[U]$-chain complex whose homology is the same as $\mathbb{H}^-(R)$. See \cite[Lemma~4.3]{DM19} for a proof of this fact.

Let $v_1, v_2, \ldots, v_n$ (resp. $\alpha_1, \alpha_2, \ldots, \alpha_{n-1}$) be the leaves (resp. the upward-opening angles) of $R$, enumerated from left to right in lexicographic order. The grading of vertex $v_i$ and the grading of the vertex supporting angle $\alpha_i$ are denoted by $\nu(v_i)$ and $\nu(\alpha_i)$, respectively; see Figure~\ref{fig:graded_root}.

\begin{figure}[htbp]
\centering
\includegraphics[width=0.35\columnwidth]{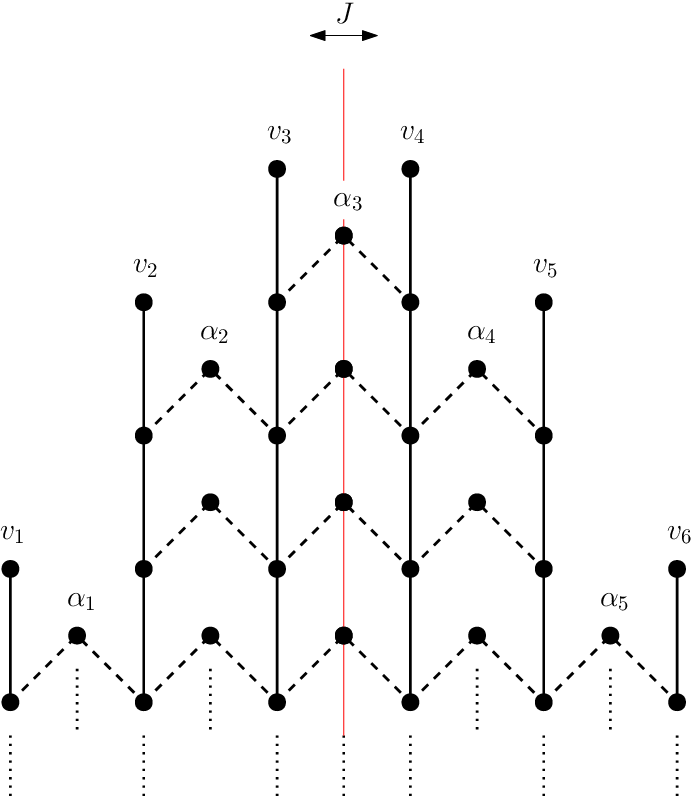}
\caption{The standard complex $C_*(R)$ with the involution $J$ that captures the lattice homology $\mathbb{H}^- (\Sigma(3,4,13))$. Here, the solid (resp. the dashed) lines represent the action of $U$ (resp. $\partial$).}
\label{fig:complex}
\end{figure}

Then the generators of $C_*(R)$ are given as follows. For each leaf $v_i$, we place a single generator (also denoted by $v_i$ by abuse of notation) in grading $\nu(v_i)$, so that we introduce an entire tower of generators $\mathbb{F}[U]v_i$. For each angle $\alpha_i$, we similarly place a single generator (denoted by also $\alpha_i$ by abuse of notation) in grading $\nu(\alpha_i) + 1$. We next define our differential to be identically zero on $v_i$, and set 
\[
\partial \alpha_i = U^{(\nu(v_i)-\nu(\alpha_i))/2}v_i + U^{(\nu(v_{i+1})-\nu(\alpha_i))/2}v_{i+1}
\]
for $\alpha_i$. Finally, we extend it to the entire complex linearly and $\mathbb{F}[U]$-equivariantly.

There is an involution $J$ on $C_*(R)$ induced by $J$ on $R$, and it is given by sending $v_i$ to $v_{n-i+1}$, $\alpha_i$ to $\alpha_{n-i}$, and extending linearly and $\mathbb{F}[U]$-equivariantly. See Figure \ref{fig:complex}. One can check that the pair $(C_*(R), J)$ forms an $\iota$-complex. Moreover, $J$ also induces the involution $J_*$ on $\mathbb{H}^-(R)$. See \cite[Section~4]{DM19} and \cite[Section~2.3]{DS19} for more details.

In \cite[Section~6]{DM19}, Dai and Manolescu introduced the notion of a \emph{monotone graded subroot} of a symmetric graded root $R$, to simply determine the local equivalence class of a Seifert fibered sphere. 

Let $(R, \nu, J)$ be a symmetric graded root, and let $\mathcal{G}_R$ be the image of the grading function $\nu: R \to \Z$. For a positive integer $n$, let $h_1, \ldots, h_n$ (resp. $r_1, \ldots, r_n$) be a sequence of decreasing (resp. increasing) even integers in $\mathcal{G}_R$ such that $h_n \geq r_n$. Then a \emph{monotone graded subroot} $$M=M(h_1,r_1; \ldots; h_n,r_n)$$ of $R$ equipped with the same involution $J$ is constructed in the following fashion:

\begin{enumerate}
\item Form the stem of $M$ by drawing a single infinite tower with the uppermost vertex in grading $r_n$,
\item If $h_n > r_n$, use leaves $v_i$ and $J v_i$ of $R$ to introduce leaves $v_i$ and $J v_i$ of $M$ in grading $h_i$ for each $1 \leq i < n$, 
\item Connect $v_i$ and $J v_i$ to the stem by using two paths meeting the stem in grading $r_i$ for each $1 \leq i < n$,
\item If $h_n = r_n$, then set $v_n = J v_n$ at grading $r_n$ in the second step.
\end{enumerate}

By construction, the standard complex $(C_*(M), J)$ of a monotone graded subroot $M$ is a subcomplex of $(C_*(R), J)$, and it also forms an $\iota$-complex itself, see \cite[Section~6]{DM19}. However, this subcomplex is sufficient to describe the local equivalence class of a Seifert fibered sphere completely, as proven by Dai and Manolescu.

\begin{theorem}[Theorems~4.5 and 6.1, \cite{DM19}]
\label{thm-local_eq_combined}
Let $Y$ be a Seifert fibered sphere with a symmetric graded root $R$. Let $M$ be a monotone graded subroot of $R$. Then the $\iota$-complexes $(CF^-(Y), \iota)$, $(C_* (R), J)$, and $(C_* (M), J)$ are locally equivalent in $\mathfrak{I}$.
\end{theorem}

We now describe the parameterization of a monotone graded subroot for a Seifert fibered sphere by using involutive correction terms $\underline{d}$ and $\overline{d}$. See Figure~\ref{fig:monotone} for an example.

\begin{theorem}[Section 8, \cite{DM19}; Theorem~4.4, \cite{DS19}]
\label{thm:parametrization}
Let $Y$ be a Seifert fibered sphere with a graded root $R$. Then there exists a monotone graded subroot $M = M(h_1,r_1; \ldots; h_n,r_n)$ of $R$ such that $$h_1 = \overline{d}(Y) \quad \text{and} \quad r_n = \underline{d}(Y).$$
\end{theorem}

\begin{figure}[htbp]
\centering
\includegraphics[width=0.3\columnwidth]{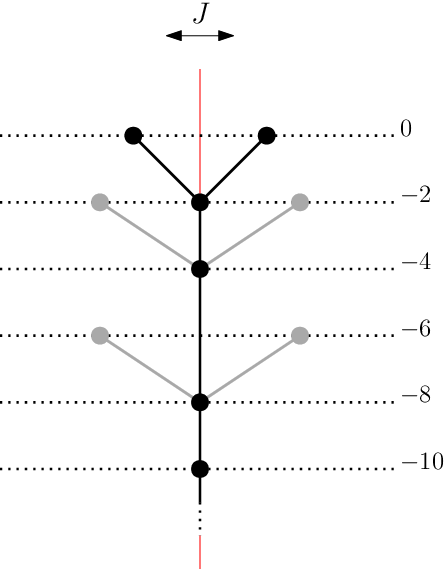}
\caption{For the Brieskorn sphere $\Sigma(3,4,13)$, the monotone graded subroot $M = M(0,-2)$ with the involution $J$ is drawn in black, compare with Figure~\ref{fig:graded_root}.}
\label{fig:monotone}
\end{figure}

By following the convention in \cite{DS19}, from now on, we will refer to the standard complex $(C_* (M), J)$ as just the monotone graded subroot $M$ for simplicity.

Now we present a theorem, provided by Dai and Stoffregen \cite{DS19}, to describe how a monotone graded subroot decomposes into simpler ones up to local equivalence.

\begin{theorem}[Theorem 4.2, \cite{DS19}]
\label{thm:decomposition}
For any monotone graded subroot $M=M(h_1,r_1; \ldots; h_n,r_n)$, we have the local equivalence: $$M = \left ( \sum_{i=1}^{n} M(h_i,r_i) \right ) - \left ( \sum_{i=1}^{n-1} M(h_{i+1},r_i) \right ) .$$
\end{theorem}

Relying on the above results, Dai and Stoffregen further proved the following theorem. Recall that $B(0) = \Sigma(2,3,5)$, and $B(n) = \Sigma(2n+1, 4n+1, 4n+3)$ for $n \geq 1$.

\begin{theorem}[Theorem~1.1, \cite{DS19}]
\label{thm:basis}
The image $\mathfrak{I}_{SF}$ of the subgroup $\Theta_{SF}^3$ generated by the Seifert fibered spheres under the homomorphism $h$ is infinitely generated, i.e., $$\mathfrak{I}_{SF} = \langle h(B(n)) \rangle \cong \Z^\infty.$$ 
\end{theorem}

For any Seifert fibered sphere $Y$, there exists a unique homology sphere $Z_Y$ expressed as the linear combination of the basis elements $\{B(n)\}$ such that $h(Y)=h(Z_Y)$, i.e., the homology sphere $Y \# -Z_Y$ lies in $\operatorname{ker}(h)$. We call such $Z_Y$ the \emph{associated homology sphere to $Y$ with respect to $\{B(n)\}$}. We omit the reference to the basis when it is clear from the context. In the final section of this article, beyond merely establishing the existence of $Z_Y$ and the linear independence of $Y$ and $Z_Y$ in $\Theta_\Q^3$, we will explicitly construct $Z_Y$ for a specific family $Y$ by applying Theorem \ref{thm:decomposition}.

\subsection{Local equivalence in Seiberg--Witten Floer theory}

Based on the work of Manolescu \cite{Man14}, Stoffregen \cite{Sto20} previously defined an analogous notion of the local equivalence class for homology spheres in Seiberg--Witten Floer theory, and computed these classes for Seifert fibered spheres. In this subsection, we present the equivalence between two notions of the local equivalence class in Seiberg--Witten Floer theory and the one in involutive Heegaard Floer theory for each Seifert fibered sphere.

As described in \cite{Sto20}, the local equivalence class of the $\operatorname{Pin}(2)$-equivariant Seiberg--Witten Floer stable homotopy type $SWF(Y)$ and of its associated chain complex of its spectrum also form abelian groups $\mathfrak{LE}$ and $\mathfrak{CLE}$, respectively, together with the homomorphisms: $$\Theta_\Z^3 \rightarrow \mathfrak{LE}\rightarrow\mathfrak{CLE},$$ where  

\begin{itemize}
    \item the $\kappa$-invariant \cite{Man14} and $\kappa o_i$-invariants \cite{Lin15} for $i = 0, \ldots, 7$ factor through $\mathfrak{LE}$;
    \item the $\delta$-invariant \cite{Fro10}, the $\alpha$-, $\beta$-, and $\gamma$-invariants \cite{Man16}, and the $\underline{\delta}$- and $\overline{\delta}$-invariants \cite{Sto17b} factor through $\mathfrak{CLE}$.
\end{itemize}

\noindent See \cite[Section~3]{Man18} for further details.

Recall that $SWF(Y)$ recovers the $\operatorname{Pin}(2)$-equivariant Seiberg--Witten Floer homology of $Y$.\footnote{Here, the Seiberg--Witten Floer homology means the $\operatorname{Pin}(2)$-equivariant Borel homology $\widetilde{H}_*^{\operatorname{Pin(2)}}(SWF(Y))$ of the space corresponding to $SWF(Y)$. By \cite{LM18, Pan25}, it turned out that this homology is isomorphic to the $\operatorname{Pin(2)}$-monopole Floer homology $\widecheck{HS_*}(Y)$ defined by \cite{Lin18}, based on \cite{KM07}.} Motivated from the facts that $\mathbb{H}^-(R)$ is isomorphic to $HF^-(Y)$ \cite{Nem05}, and hence to the Seiberg--Witten Floer homology \cite{LM18,KLT20}, Dai, Sasahira, and Stoffregen constructed the $\operatorname{Pin}(2)$-equivariant spectrum $\mathcal{H}(R)$ of the graded root $R$ of $Y$, and proved the following:

\begin{theorem}[Theorems~1.2 and 7.3, \cite{DSS23}]
\label{thm:DSS}
  Let $Y$ be a Seifert fibered sphere with a graded root $R$. Then there is a $\operatorname{Pin}(2)$-equivariant homotopy equivalence:
  $$\mathcal{H}(R)\simeq SWF(Y).$$
  In particular, for two Seifert fibered spheres $Y_1$ and $Y_2$, the following are equivalent:
  \begin{enumerate}
    \item $h(Y_1)=h(Y_2)$ in $\mathfrak{I}$.
    \item $[SWF(Y_1)] = [SWF(Y_2)]$ in $\mathfrak{LE}$.
  \end{enumerate}
\end{theorem}

In particular, if the associated homology sphere $Z_Y$ to a Seifert fibered sphere $Y$ is itself a Seifert fibered sphere, then the homology sphere $Y \# -Z_Y$ in $\ker (h)$ also lies in the kernel of $\Theta^3_\Z \to \mathfrak{LE}$. 

\subsection{Filtered instanton Floer homology and \texorpdfstring{$r_s$}{rs}-invariants}

In this subsection, we briefly recall instanton Floer homology and its filtered version by Nozaki, Sato, and Taniguchi \cite{NST24}. We follow the conventions and notations in \cite{NST24}.

For a given homology sphere $Y$, the Chern--Simons functional $cs_Y$ is defined on the space $A(Y)$ of $SU(2)$-connections on the product principal $SU(2)$-bundle $P_Y = Y \times SU(2)$.\footnote{We write $=$ rather than $\cong$ to emphasize the choice of the trivialization of the product bundle $P_Y$. This makes the choice of the reducible critical point $\theta$, which will be defined later, canonical.} Since $cs_Y$ is invariant under the action of the group $\operatorname{Map}_0(Y, SU(2))$ consisting of degree $0$ maps in the gauge group, $cs_Y:A(Y)\rightarrow \R$ factors through the configuration space $\widetilde{B}(Y) = A(Y) / \operatorname{Map}_0(Y, SU(2))$. We consider the subset $\widetilde{B}^*(Y)$ which consists of the orbits of \emph{irreducible} connections, namely the stabilizer subgroup is $\{\pm I\}$ where $I$ denotes the constant map to the identity in $SU(2)$.

We fix a Riemannian metric on $Y$ and an appropriate perturbation $\pi$ of $cs_Y$. Then we construct an infinite-dimensional analogue of Morse chain complex of $\widetilde{B}^*(Y)$ using the perturbed Chern--Simons functional $cs_{Y,\pi}$. The set $\widetilde{R}^*_\pi(Y)$ of the irreducible critical points of $cs_{Y,\pi}$ in $\widetilde{B}^*(Y)$, consists of \textit{flat} connections, namely, those with curvature $F_A=0$. The integer-valued Floer index $ind$, which is an analogue of the ordinary Morse index, is also well-defined on $\widetilde{R}^*(Y)$. Then the \emph{instanton Floer chain complex} is defined as:
$$CI_i(Y)=\Z\{a \in \widetilde{R}_\pi^*(Y) \ | \ ind(a) = i\},$$
together with the boundary map $\partial: CI_i(Y) \rightarrow CI_{i-1}(Y)$ given by:
$$\partial (a) = \sum\limits_{\mathclap{\substack{
b \in \widetilde{R}^*_\pi(Y) \\
\mathrm{ind}(b) = i-1}}}\#(M^Y(a, b)/\R)b.$$ Here, $M^Y(a, b)$ denotes the $1$-dimensional moduli space of instantons over the cylinder $\R \times Y$ where its restriction on $Y\times t$ is asymptotically $a$ when $t$ goes to $-\infty$ and $b$ when $t$ goes to $+\infty$. Note that $\R$ acts on $M^Y(a,b)$ by the translation.

The boundary map satisfies $\partial^2 = 0$, so we obtain the \emph{instanton Floer homology} $I_*(Y)$, which does not depend on the choice of Riemannian metric and perturbation, and is a diffeomorphism invariant of $Y$. By taking the dual complex $CI^i(Y)$, one can also obtain the \emph{instanton Floer cohomology} $I^*(Y)$.

Note that the chain complex only involves irreducible critical points. Let $\theta \in A(Y)$ be the product $SU(2)$-connection on $P_Y$ so that $\theta$ is a reducible critical point. Donaldson \cite{Don02} introduced an obstruction cocycle $\theta_Y:CI_1(Y)\rightarrow \Z$ in the first instanton Floer cochain complex $CI^1(Y)$ to count the trajectories from irreducible solutions to $\theta$ as follows:
$$\theta_Y([a]) = \#(M^Y(a, \theta)/\R),$$
where $M^Y(a,\theta)$ is similarly defined to $M^Y(a,b)$. See \cite{NST24} for a precise definition. The cocycle $\theta_Y$ forms a well-defined cohomology class $[\theta_Y]\in I^1(Y)$, which vanishes whenever $Y$ bounds a negative definite $4$-manifold.

On the other hand, Fintushel and Stern \cite{FS90} introduced the filtered instanton Floer homology $I_*^{[r,r+1]} (Y)$ for any fixed real number $r$ whose filtrations are given by the Chern--Simons functional. We recall the recent work of Nozaki, Sato, and Taniguchi \cite{NST24} which extends the work of Fintushel and Stern to the filtered instanton Floer homology $I_*^{[s,r]} (Y)$ for a more general interval $[s,r]$. We write the set of real numbers:
$$\Lambda_Y = cs_Y(\widetilde{R}_\pi(Y)), \quad \Lambda_Y^* = cs_Y(\widetilde{R}^*_\pi(Y)), \quad \text{and} \quad  \R_Y = \R - \Lambda_Y,$$ where $\widetilde{R}_\pi(Y)$ is the set of the orbits of all critical points of $cs_{Y,\pi}$, whether irreducible or reducible.

\begin{definition}
  Let $s\in [-\infty, 0]$, $r\in \R_Y$ and $\lambda_Y = \min\{|a-b| |a \neq b, a, b\in \Lambda_Y\}$. The \emph{filtered instanton Floer chain complex} is defined as:
  $$CI_i^{[s, r]}(Y) = \begin{cases}
    \Z\{a \in \widetilde{R}_{\pi}^*(Y) \ | \ ind(a) = i, cs_{Y, \pi}(a)\in (s, r)\} & \text{if $s\in \R_Y$,}\\
    \Z\{a \in \widetilde{R}_{\pi}^*(Y) \ | \ ind(a) = i, cs_{Y, \pi}(a)\in (s-\lambda_Y/2, r)\}& \text{if $s\in \Lambda_Y$,}
  \end{cases}$$ together with the boundary map given by the restriction of $\partial$ defined on the ordinary $CI_i(Y)$.
\end{definition}

\noindent Since $\partial^2=0$, we have the \textit{filtered instanton Floer homology} $I_*^{[s,r]}(Y)$, which is a diffeomorphism invariant of $Y$. Technically, the chain complex can be defined for arbitrary intervals, however, we avoid the case when $r\in \Lambda_Y$ to make the canonically defined map $i_{[s, r]}^{[s',r']}:CI_i^{[s,r]}(Y)\rightarrow CI_i^{[s',r']}(Y)$ behave well. More precisely,

\begin{lemma}[Lemma 2.9, \cite{NST24}]
\label{lemma-shift}
  Suppose $s\le s' \le 0 \le r \le r'$ where $r, r'\in \R_Y$. If both $[r, r']$ and $[s, s']$ do not contain any value of $\Lambda_Y^*$, then the map
    $i_{[s, r]}^{[s',r']}:CI_i^{[s,r]}(Y)\rightarrow CI_i^{[s',r']}(Y)$ defined by
    $$ i_{[s,r]}^{[s',r']}(a)= \begin{cases}
      a &\text{if $a\in CI_i^{[s',r']}(Y)$,}\\
      0 &\text{otherwise,}
    \end{cases}$$
    is a chain homotopy equivalence.
\end{lemma}

By taking duals of $CI_i^{[s,r]}(Y)$ and of the chain maps, we can also consider the \textit{filtered instanton Floer cochain complex} $CI^i_{[s,r]}(Y)$ and its cohomology $I^i_{[s,r]}(Y)$, and the dual maps on them.

Let $Y_1$ and $Y_2$ be two homology spheres. Now we consider an oriented cobordism $W$ from $Y_1$ to $Y_2$. As in other Floer theories, $W$ induces a chain map between the instanton Floer chain complexes. In particular, when $b_1(W)=b_2^+(W)=0$, the (co)chain maps preserve the Floer indices: \begin{align*}
    CW_i^{[s, r]}&:CI_i^{[s,r]}(Y_1)\rightarrow CI_i^{[s,r]}(Y_2);\\
    CW^i_{[s, r]}&:CI^i_{[s,r]}(Y_2)\rightarrow CI^i_{[s,r]}(Y_1).
\end{align*}

Let $\theta_Y^{[s,r]}$ be the restriction of $\theta_Y$ to $CI_1^{[s,r]}(Y)$. In \cite{NST24}, Nozaki, Sato, and Taniguchi observed that the map $\theta_Y^{[s,r]}$ defines a cocycle in $CI^1_{[s,r]}(Y)$. The cohomology class $[\theta_Y^{[s,r]}]\in I^1_{[s,r]}(Y)$ also behaves well under the induced map $IW^1_{[s,r]}$ by $CW^1_{[s,r]}$ on the first cohomology. More precisely,

\begin{theorem}[Lemma 2.13, \cite{NST24}]
\label{lemma-theta-cobordism}
  If there is a negative definite cobordism $W$ from $Y_1$ to $Y_2$ with $H^1(W;\R)=0$, then for $r\in \R_{Y_1}\cap \R_{Y_2}$,
  $$IW^1_{[s,r]}([\theta_{Y_2}^{[s,r]}]) = |H_1(W;\Z)|[\theta_{Y_1}^{[s,r]}].$$ In particular, if $Y$ bounds a negative definite $4$-manifold, then the obstruction class $[\theta_Y^{[s, r]}]\in I^1_{[s,r]} (Y)$ vanishes.
  \label{thm-IW-theta}
\end{theorem}

Let $i_{[s',r']}^{[s,r]}$ be the dual map of $i^{[s',r']}_{[s,r]}$. We use the same notation for the induced maps on (co)homology by abusing notation. One can observe from Lemma \ref{lemma-shift} that:

\begin{lemma}[Lemma 2.15, \cite{NST24}]
\label{lemma-theta-shift}
  Suppose $s\le s' \le 0 \le r \le r'$ where $r, r'\in \R_Y$. Then,
  $$i_{[s',r']}^{[s,r]}[\theta_Y^{[s',r']}] = [\theta_Y^{[s,r]}].$$
\end{lemma}

Suppose $[\theta_Y^{[s, r]}]\neq 0$. By definition of $i_{[s,r]}^{[s',r']}$ and Lemma \ref{lemma-theta-shift}, $[\theta_Y^{[s, r']}]$ vanishes for some $r'< r$ and then $[\theta_Y^{[s, r'']}]=0$ for any $r''\le r'$. This \say{birth-death} property motivates to define the $r_s$-invariant to measure when the obstruction class $[\theta_Y^{[s,r]}]$ begins to vanish. The coefficient ring is taken as $\Q$, since one may want to divide $IW^1_{[s,r]}([\theta_{Y_2}^{[s,r]}])$ by the non-zero integer $|H_1(W;\Z)|$ to obtain $[\theta_{Y_1}^{[s,r]}]$ as in Theorem \ref{thm-IW-theta}.

\begin{definition}[Definition 3.2, \cite{NST24}]
    Let $s\in [-\infty, 0]$. The $r_s$-invariant of $Y$ is defined as:
    $$r_s (Y) = \sup \{ r \in (0,\infty] \ | \ [\theta_Y^{[s, r]}\otimes id_\Q] = 0 \in I^1_{[s,r]} (Y;\Q) \}.$$ 
  \end{definition}

By definition and the vanishing property of the obstruction class, if $Y$ bounds a negative definite $4$-manifold, then $r_s(Y) = \infty$ for all $s\in[-\infty, 0]$. In particular, for any $s\in[-\infty, 0]$ and any Seifert fibered sphere $\Sigma(a_1,\ldots, a_n)$, we have
$$r_s(\Sigma(a_1,\ldots,a_n))=\infty.$$

Nozaki, Sato, and Taniguchi \cite[Theorem 5.15]{NST24} proved that the value $r_s(Y)$ is invariant under rational homology cobordism. They also obtained the following connected sum formula \cite[Theorem~1.1(4)]{NST24} $$r_{s_1+s_2}(Y_1\#Y_2)\ge \min\{r_{s_1}(Y_1)+s_2, \ r_{s_2}(Y_2)+s_1\},$$ so it is useful to take $s=0$ to obtain linear independence in $\Theta_\Q^3$.

\begin{theorem}[Corollary 5.6, Theorem 5.15, \cite{NST24}]
\label{thm:linear_independence}
Let $\{Y_i\}_{i\geq 1}$ be an infinite family of homology spheres. If all $r_0(Y_i)$ are distinct and finite, and $r_0(-Y_i)=\infty$, then $\{Y_i\}_{i\geq 1}$ is linearly independent in $\Theta_\Q^3$.
\end{theorem}

Theorem \ref{thm:linear_independence} will play a crucial role in the proof of Theorem \ref{thm-Z}, coupled with the useful formula for $r_0(Y)$ for a Seifert fibered sphere $Y$ such that $R(Y) > 0$, where $R$ is the integer defined for a Seifert fibered sphere by Fintushel and Stern \cite{FS85}. We close this section by presenting such a formula.

\begin{theorem}[Corollary 1.4, \cite{NST24}]
\label{thm:formula}
  Let $Y$ be the Seifert fibered sphere $\Sigma(a_1,\ldots, a_n)$. If $R(Y) >0$, then for any $s\in [-\infty, 0]$,
  $$r_s(-Y) = \frac{1}{4a_1\cdots a_n} \cdot$$
\end{theorem}

%-----Section 3-----

\section{Proof of Theorem \ref{thm-Z} and Examples}

In this section, we prove Theorem~\ref{thm-Z} which establishes conditions on a Seifert fibered homology sphere $Y$ to ensure that $Y$ and the associated homology sphere $Z_Y$ are different in $\Theta^3_{\mathbb{Q}}$, in terms of the product $\mathcal{E} (Y)$ of its exponents and the Fintushel--Stern $R$-invariant~\cite{FS85}. It is evident that the homology sphere $Y \# -Z_Y$ lies in $\operatorname{ker}(h)$, and hence provides a potential counterexample for Question~\ref{question-1}. We restate Theorem \ref{thm-Z} in a slightly stronger way which clearly implies the version in Section \ref{sec:intro}.

\begin{reptheorem}{thm-Z}
    Let $\{Y(n)\}$ be a family of Seifert fibered sphere and let $Z(n)$ be the associated homology spheres to $Y(n)$ of the same image in $\mathfrak{I}$. If~$R(Y(n)) > 0$, $\mathcal{E} (Y(n)) > 30$, and the integers $\mathcal{E}(Y(n))$ and $(2m+1)(4m+1)(4m+3)$ are all distinct, then $\{Y(n)\} \cup \{Z(n)\}$ are linearly independent in $\Theta_\Q^3$.
\end{reptheorem}

\begin{proof}
Since $Z(n)$ is a linear combination of $\{B(m)\}$, it is enough to prove that $\{Y(n)\}$ and $\{B(m)\}$ are linearly independent in $\Theta_\Q^3$. Because every Seifert fibered sphere bounds a negative definite $4$-manifold, by Theorem \ref{lemma-theta-cobordism} and the definition of the $r_s$-invariants, $$r_0(Y(n))=r_0(B(m))=\infty.$$

By the assumption $R(Y(n)) > 0$, we can apply Theorem \ref{thm:formula} to obtain $$r_0(-Y(n)) = \frac{1}{4 \mathcal{E}(Y(n))} \cdot$$ Thus, under the assumptions that $\mathcal{E}(Y(n))$ and $(2m+1)(4m+1)(4m+3)=\mathcal{E}(B(m))$ are finite and distinct, by Theorem \ref{thm:linear_independence}, the proof is complete.
\end{proof}

Now we present explicit examples of the associated homology spheres $Z_1 (n)$, $Z_2 (n)$, and $Z_3 (n)$ to the following families $Y_1 (n)$, $Y_2 (n)$, and $Y_3 (n)$ of certain Seifert fibered spheres: 

\begin{itemize}
    \item $Y_1(n) = \Sigma(4n+1, 6n+2, 12n+1)$,
    \item $Y_2(n) = \Sigma(4n-1, 6n-2, 12n-1)$,
    \item $Y_3 (n) = \Sigma(8n+1, 12n+1, 24n+5)$,
\end{itemize}

\noindent respectively, in terms of $B (n)$ using Theorem~\ref{thm:decomposition}. Then we prove linear independence of them in $\Theta_\Q^3$ by using Theorem \ref{thm-Z}. 

It is well-known that the graded root of the Poincar\'e homology sphere has the monotone graded subroot $M(2,2)$ since $\overline{d}(\Sigma(2,3,5)) = 2 = \underline{d}(\Sigma(2,3,5))$. By Theorem~\ref{thm:parametrization}, $M(2,2)$ is a single tower with the uppermost vertex in grading $2$.

We next describe the monotone graded subroots of $B(n)$ for $n \geq 1$ \cite{DS19} (see also \cite{Sto17}) and the ones for the families $Y_1 (n)$, $Y_2 (n)$, and $Y_3 (n)$ \cite{KS22} (see also \cite{KS20}):

\begin{theorem}[Theorem~1.1, \cite{DS19}; Theorem~1.2, \cite{KS22}]
\label{thm:monotone_roots}
For $n\geq 1$, the Brieskorn spheres $B(n)$, $Y_1(n)$ and $Y_2(n)$ have the monotone graded subroots $$M(2n,0), \quad M(4n,0;2n,2n), \quad \text{and} \quad M(4n-2,0;2n,2n),$$ respectively, which are depicted in Figure~\ref{fig:roots}. Moreover, the Brieskorn spheres $Y_3 (n)$ has the same monotone graded subroot as $B(4n)$.
\end{theorem}

\begin{figure}[htbp]
\centering
\includegraphics[width=0.7\columnwidth]{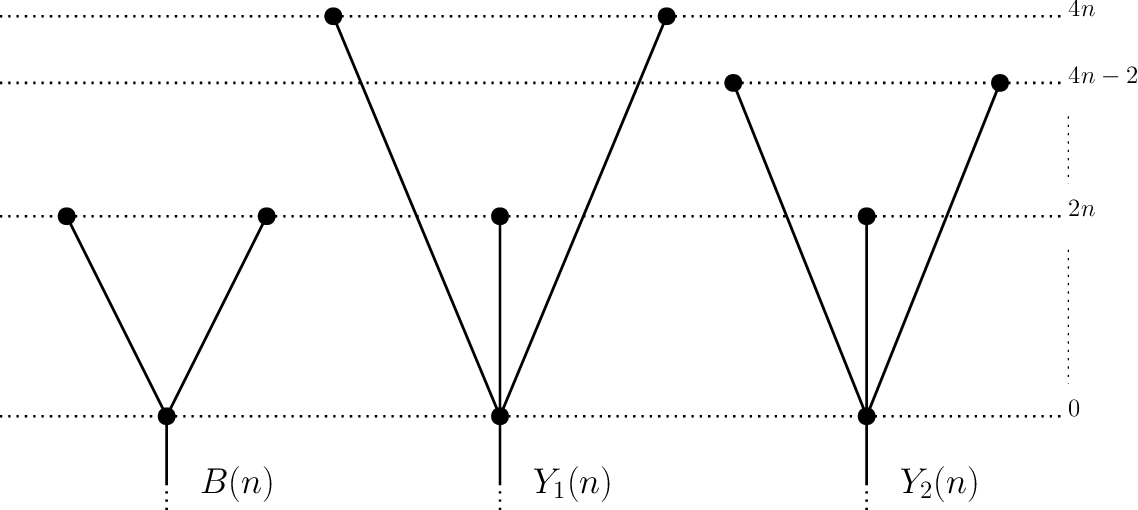}
\caption{The monotone graded subroots of $B(n)$, $Y_1(n)$ and $Y_2(n)$ for $n \geq 1$.}
\label{fig:roots}
\end{figure}

\begin{lemma}
\label{lemma:key_lemma2}
In the local equivalence group $\mathfrak{I}$, we have the following for $n \ge 1:$
\begin{enumerate}
    \item \label{case2} $h(Y_1(n)) = h(B(2n) \ \# \ - B(n) \ \# \ nB(0))$,
    \item \label{case3} $h(Y_2(n)) = h(B(2n-1) \ \# \ - B(n) \ \# \ nB(0))$,
    \item \label{case4} $h(Y_3(n)) = h(B(4n))$.
\end{enumerate}
\end{lemma}

\begin{proof}
By applying Theorems~\ref{thm-local_eq_combined} and \ref{thm:monotone_roots} together, we first see that $$h(Y_1 (n)) = M(4n,0; 2n, 2n) \quad \text{and} \quad h(Y_2 (n)) = M(4n-2,0; 2n, 2n).$$

Note that $M(2n, 2n)$ is the single tower with the uppermost vertex in grading $2n$, and it is obtained by tensoring $n$ copies of $M(2,2)$, where $M(2,2)=h(B(0))$. So we have: $$M(2n,2n) = \underbrace{M(2,2) + \cdots + M(2,2)}_{n\text{ copies}}.$$

We next use Theorem~\ref{thm:decomposition} to decompose the monotone graded subroots above. Since $h: \Theta^3_\Z \to \mathfrak{I}$ is a homomorphism, for the case (\ref{case2}), we have:
\begin{align*}
    h(Y_1 (n)) & = M(4n,0;2n,2n), \\
    &= M(4n,0) + M(2n,2n) - M(2n,0), \\
    &= M(4n,0) +  \underbrace{M(2,2) + \cdots + M(2,2)}_{n\text{ copies}} - M(2n,0), \\
    &= h(B(2n) \ \# \ - B(n) \ \# \ nB(0)).
\end{align*} One can similarly prove the case (\ref{case3}). Finally, the case (\ref{case4}) is a direct consequence of Theorem~\ref{thm:monotone_roots}.
\end{proof}

Therefore, the associated homology spheres $Z_1(n)$, $Z_2(n)$, and $Z_3(n)$ respectively to $Y_1(n)$, $Y_2(n)$, and $Y_3(n)$ are as follows:

\begin{itemize}
    \item $Z_1 (n) = B(2n) \ \# \ - B(n) \ \# \ nB(0)$,
    \item $Z_2 (n) = B(2n-1) \ \# \ - B(n) \ \# \ nB(0)$,
    \item $Z_3 (n) = B(4n)$.
\end{itemize}

As an application of Theorem~\ref{thm-Z}, we will show that $\{ Y_i(n)\#-Z_i(n) \}_{n \geq 1}$ for $i=1,2,3$, which lie in the kernel of $h$, are linearly independent in $\Theta_\Q^3$.

\begin{comment}
Next, we will show that all the Brieskorn spheres $B(n)$, $Y_1(n)$, and $Y_2(n)$ are linearly independent up to rational homology cobordism. In particular, the homology spheres in Theorem \ref{thm-Z} are generated by the above.
\end{comment}

\begin{theorem}
\label{theorem:key_lemma1}
The homology spheres
$$\big \{ Y_1 (n) \# - Z_1 (n) \big \}_{n \geq 1} \ \cup \ \big \{ Y_2 (n) \# - Z_2 (n) \big \}_{n \geq 1}$$ are linearly independent in $\operatorname{ker}(h)$ and $\Theta_\Q^3$.
\end{theorem}

\begin{proof}

By \cite[Proposition~4.3]{KS20}, all of these homology spheres have negative definite plumbing graphs with central weights $e = -2$. Thus, since $R(Y)=-2e-3$ for any Seifert fibered sphere $Y$ \cite{NZ85}, we have: $$R(B(n)) = R(Y_1(n)) = R(Y_2(n)) = 1.$$

To apply Theorem \ref{thm-Z}, it is enough to show that all $\mathcal{E} (B(n))$, $\mathcal{E}(Y_1(m))$ and $\mathcal{E}(Y_2(k))$ for any positive integers $n, m$ and $k$ are distinct. It is clear that $(2n+1)(4n+1)(4n+3)$, $(4m+1)(6m+2)(12m+1)$ and $(4k-1)(6k-2)(12k-1)$ are never the same as $2\cdot 3 \cdot 5 = 30$.

Note that $(2n+1)(4n+1)(4n+3)$ is always odd. In contrast, we see that $(4m\pm 1)(6m\pm 2)(12m\pm 1)$ are always even. Thus, there are no pair of positive integers $(n, m)$ such that $(2n+1)(4n+1)(4n+3) = (4m\pm 1)(6m\pm 2)(12m\pm 1)$.

Finally, by dividing both of $(4m\pm 1)(6m\pm 2)(12m\pm 1)$ by $2$, define $f(a) = (4a + 1)(3a+1)(12a+1)$ and $g(b) = (4b-1)(3b-1)(12b-1)$. Then for any $a \geq 1$, we can see that $g(a) < f(a) < g(a+1)$. Since $g(b)$ is strictly increasing when $b \geq 1$, there cannot exist a pair of positive integers $(a,b)$ such that $f(a) = g(b)$, which implies that $\mathcal{E}(Y_1(m))\neq \mathcal{E}(Y_2(k))$, which completes the proof.
\end{proof}

On the other hand, consider $Y_3(n)$ and the associated one $Z_3(n)$. In this case, $Z_3(n)$ is also Seifert fibered, so by Theorem~\ref{thm:DSS}, we see that $\{ Y_3 (n) \# - Z_3 (n) \}_{n \geq 1}$ have the trivial local equivalence class of $\operatorname{Pin(2)}$-equivariant Seiberg--Witten Floer stable homotopy type in $\mathfrak{LE}$. However, we also prove their linear independence in $\Theta_\Q^3$.

\begin{theorem}
\label{theorem:key_lemma1.5}
The homology spheres
$$\big \{Y_3(n) \# - Z_3 (n) \big \}_{n \geq 1} $$ are linearly independent in $\operatorname{ker}(\Theta^3_\Z \to \mathfrak{LE})$ and $\Theta_\Q^3$.
\end{theorem}
\begin{proof}
    For the same reason as in the proof of Theorem ~\ref{theorem:key_lemma1}, we find that $R(Y_3(n)) = 1$. Since $$(8n+1)(16n+1)(16n+3) \not \equiv (8m+1) (12m+1)(24m+5) \mod 4,$$ we conclude that $\mathcal{E} (Y_3 (n))$ and $\mathcal{E} (Z_3 (m))$ are distinct and we can apply Theorem~\ref{thm-Z}.
\end{proof}

Note that our examples in Theorem~\ref{theorem:key_lemma1} give other families of homology spheres whose local equivalence classes (also of Seiberg--Witten Floer types for Theorem~\ref{theorem:key_lemma1.5}) are trivial, yet which are linearly independent in $\Theta^3_\Q$, as stated in Corollary~\ref{cor-A}.

\begin{comment}

Finally, we prove that the following homology spheres are linearly independent

$$ \big \{Y_1(n) \ \# \ -B(2n) \ \# \ B(n) \ \# \ -n B(0)  \big \}_{n \geq 1} \cup \big \{Y_2(n) \ \# \ -B(2n-1) \ \# \ B(n) \ \# \ -n B(0)  \big \}_{n \geq 1} \cdot$$

By (1) and (2) in Lemma~\ref{lemma:key_lemma2}, we know that they lie in the kernel of the homomorphism $h: \Theta^3_\Z \to \mathfrak{I}$. Applying Lemma~\ref{lemma:key_lemma1}, we also see that they are linearly independent in $\Theta^3_\Q$, which completes the proof.

Consider the other family of homology spheres $$\{Y_3(n)\#-B(4n)\}_{n \geq 1}.$$ By (3) in Lemma~\ref{lemma:key_lemma2}, they also lie in $\ker{h}$. Moreover, since $h(Y_3(n))=h(B(4n))$, by Theorem \ref{thm:DSS}, above homology spheres are of the trivial local equivalence class of $\operatorname{Pin(2)}$-equivariant Seiberg--Witten Floer stable homotopy type in $\mathfrak{LE}$. While they are lying in 
$\ker\{\Theta^3_ \Z \xrightarrow{h} \mathfrak{I}\} \cap \ker\{\Theta^3_ \Z \rightarrow\mathfrak{LE}\}$, by applying Lemma~\ref{lemma:key_lemma1.5}, they are still linearly independent in $\Theta_\Q^3$.

\end{comment}

\bibliography{main.0.0}
\bibliographystyle{amsalpha}
\end{document}